\documentclass{amsart}

\usepackage[latin1]{inputenc}
\usepackage{comment}
\usepackage{qtree,array,youngtab}

\usepackage{color, graphicx}
\usepackage{soul}
\usepackage[T1]{fontenc}

\definecolor{darkgreen}{rgb}{0.,0.5,0.}

\numberwithin{equation}{section}
\newtheorem{theorem}{Theorem}[section]
\newtheorem{lemma}[theorem]{Lemma}

\newtheorem{ex}[theorem]{Example}
\newtheorem{rem}[theorem]{Remark}
\allowdisplaybreaks

\newcommand{\la}{\lambda}
\allowdisplaybreaks

\title{Core Partitions with Distinct Parts}

\author{Huan Xiong}
\address{Universit\'e de Strasbourg, CNRS, IRMA UMR 7501, F-67000 Strasbourg, France}
\email{xiong@math.unistra.fr}



\begin{document}
\begin{abstract} Simultaneous core partitions have attracted much attention since Anderson's work on the number of $(t_1,t_2)$-core partitions. In this paper we focus on simultaneous core partitions with distinct parts. The generating function of $t$-core partitions with distinct parts is obtained. We also prove the results on the number, the largest size and the average size of $(t, t + 1)$-core partitions. This gives a complete answer to a conjecture of Amdeberhan, which is partly and independently proved by Straub, Nath and Sellers, and Zaleski recently.
\end{abstract}

 \maketitle

\section{Introduction}
The aim of this paper is to study simultaneous core partitions with distinct parts.  
Let's recall some basic definitions first. We refer the reader to \cite{Macdonald,ec2} for the basic knowledge on partitions.
A \emph{partition} is a finite
nonincreasing  sequence $\lambda =
(\lambda_1, \lambda_2, \ldots, \lambda_\ell)$  of positive integers. Here $\lambda_i\ (1\leq i\leq \ell)$ are called the parts of $\lambda$ and $| \lambda
|=\sum_{1\leq i\leq \ell}\lambda_i$ is the \emph{size} of $\lambda$.
A partition $\lambda$ is usually identified with its \emph{Young diagram},
which is a collection of left-justified rows  with
$\lambda_i$ boxes in the $i$-th row.  The \emph{hook length} of  the box $\square=(i,j)$ in the $i$-th row and $j$-th column of the Young diagram,  denoted by $h{(i, j)}$, is  the
number of boxes exactly to the right, or exactly below, or the box
itself. For example, Figure \textbf{$1$} shows the Young diagram and hook
lengths of the partition $(5, 3,3, 2,1)$.
For positive integers $t_1,t_2,\ldots, t_m$, a partition is 
called a \emph{$(t_1,t_2,\ldots, t_m)$-core partition} if it doesn't have  hook lengths in  $\{t_1,t_2,\ldots, t_m\}$. In particular, a partition is 
called a \emph{$t$-core partition} if it doesn't have the hook length $t$ (see \cite{stanton2,stanton}).  For example, we can see from Figure \textbf{$1$} that $\lambda=(5, 3,3, 2,1)$ is a $(8,10)$-core partition.

\begin{figure}[htbp]
\begin{center}
\Yvcentermath1

\begin{tabular}{c}
$\young(97521,642,531,31,1)$

\end{tabular}

\end{center}
\caption{The Young diagram of the partition $(5, 3,3, 2,1)$ and the hook
lengths of corresponding boxes. }
\end{figure}

Many results have been obtained in the study of 
$(t_1,t_2,\ldots, t_m)$-core partitions.  For $m=2,$ Anderson \cite{and} showed that the number of $(t_1,t_2)$-core
partitions is the \emph{rational Catalan number} $\frac{1}{t_1+t_2}  \binom{t_1+t_2}{t_1}$ when $t_1$
and $t_2$ are coprime to each other. 
Olsson and Stanton \cite{ols} found the largest size of
such partitions, which is $ \frac{(t_1^2-1)(t_2^2-1)}{24}$.
Various results on the enumeration of $(t_1,t_2)$-core partitions are achieved by \cite{AHJ, CHW,  PJ, N3, SZ,Wang}. 
A specific type of simultaneous core partitions, $(t,t+1,\ldots, t+p)$-core partitions, had been well studied. Results on the  number, the largest size and the average size of such partitions can be found in \cite{tamd1, NS, Xiong1,  YZZ}. %Recently, the number of $(t,t+d,t+2d)$-core partitions for coprime $t,d\ge1$,  conjectured by Amdeberhan and  Leven \cite{tamd1}, was verified by  Wang \cite{Wang} and Johnson (private communication to Amdeberhan).  

Much attention has been attracted to simultaneous core partitions with distinct parts since Amdeberhan's conjectures \cite{tamd} on this subject in $2015$. The results on the enumeration of $(t,t+1)$, $(t,t+2)$ and $(t,nt\pm1)$-core partitions with distinct parts can be found in several papers \cite{NS,Straub,YQJZ,ZZ,Za} published in $2016$ and $2017$. 
In this paper\footnote{
We mention that our paper is one of the earliest papers in the study of simultaneous core partitions with distinct parts (the first edition of this paper was available on arxiv since August 2015, which is cited by the above all five papers).},
we obtain the generating function of $t$-core partitions with distinct parts in Theorem \ref{distinct*}. We also prove the results on the number, the largest size and the average size of $(t, t+1)$-core partitions with distinct parts in Theorem \ref{main},  which verify Amdeberhan's conjecture on such partitions. Notice that part of Theorem \ref{main} was generalized independently by \cite{NS2,Straub,Za}. In fact, Straub \cite{Straub} and Nath and Sellers \cite{NS2} found the number of $(t, nt-1)$ and $(t, nt+1)$-core partitions with distinct parts respectively. Zaleski \cite{Za} obtained the explicit expressions for the moments of the sizes of $(t,t+1)$-core partitions with distinct parts, which gave a generalization of Theorem \ref{main} $(4)$. Furthermore, Zaleski and Zeilberger \cite{ZZ} also obtained the the moments of the sizes of $(2t+1,2t+3)$-core partitions with distinct parts, whose number,  largest size and average size were given by Yan, Qin, Jin and Zhou \cite{YQJZ}. Our main results are stated next.

\begin{theorem} \label{distinct*}
Suppose that $t\geq 2$. Let $cd_t(n)$ be the number of $t$-core partitions of size $n$ with distinct parts. Then the generating function for such partition is  
\begin{equation}\label{eq:2}
\sum_{n\geq 0}cd_t(n)q^n=\sum_{(n_1,n_2,\cdots,n_{t-1})\in \mathcal{C}_t}q^{\sum_{i=1}^{t-1}\left(in_i+t\binom{n_i}{2}\right)-\binom{\sum_{i=1}^{t-1}n_i}{2}},
\end{equation}
where $\mathcal{C}_t=\{ (x_1,x_2,\cdots,x_{t-1})\in\mathbb{N}^{t-1}: x_ix_{i+1}=0 \ \text{for}\ 1\leq i\leq t-2  \}$.
In particular, when $t=2,3,4$, we have
$$
\sum_{n\geq 0}cd_2(n)q^n=\sum_{n\geq 0}q^{\binom{n+1}{2}},
$$
$$
\sum_{n\geq 0}cd_3(n)q^n=\sum_{n\geq 1}q^{n^2}+\sum_{n\geq 0}q^{2\binom{n+1}{2}},
$$ 
and
$$\sum_{n\geq 0}cd_4(n)q^n=\sum_{n\geq 1}q^{\frac{n(3n+1)}2}+\sum_{n\geq 0}\sum_{m\geq 0}q^{\frac{n(3n-1)+3m(m+1)-2mn}2}.$$
\end{theorem}

\begin{theorem}[Cf. Conjecture 11.9 of  \cite{tamd}] \label{main} 
Let $t\geq 2$ be a positive integer and $(F_i)_{i\geq 0}=(0,1,1,2,3,5,8,13,\ldots)$ be the  Fibonacci numbers.  For $(t, t + 1)$-core partitions with distinct parts, we have
the following results.

$(1)$ The number of such partitions is  $F_{t+1}.$

$(2)$ The largest size of such partitions is 
$[\frac13\binom{t+1}{2}]$, where  $[x]$ is the largest integer not greater than $x$.

$(3)$ The number of such partitions with the largest size is $2$ if $t  \equiv 1 \ (\text{mod}\ 3)$ and $1$ otherwise.

$(4)$ The total sum of the sizes of these partitions and the average size are, respectively, given by
$$\sum\limits_{\substack{i+j+k=t+1\\ i,j,k\geq 1}}F_iF_jF_k\ \ \ \text{and}\ \  \ \sum\limits_{\substack{i+j+k=t+1\\ i,j,k\geq 1}}\frac{F_iF_jF_k}{F_{t+1}}.$$
\end{theorem}

\section{The $\beta$-sets of core partitions}

In this section, we study the properties of $\beta$-sets of $t$-core partitions and obtain the generating function for $t$-core partitions of size $n$ with distinct parts.

Suppose that $\lambda = (\lambda_1, \lambda_2, \ldots, \lambda_\ell)$ is a partition whose corresponding Young diagram has $\ell$ rows.
 The \emph{$\beta$-set} $\beta(\lambda)$ of $\lambda$ is defined to
 be the set of
 \emph{first-column hook lengths} in the Young diagram of $\lambda$ (for example, see \cite{ols,Xiong1}), i.e., $$\beta(\lambda)=\{h(i,1) : 1 \leq i \leq \ell\}.$$  The following results are well-known and easy to
prove.

\begin{lemma}[\cite{and, berge,ols,Xiong1}]  \label{betaset}

(1) Suppose $\lambda = (\lambda_1, \lambda_2, \ldots, \lambda_\ell)$ is a partition.
Then $\lambda_i=h(i,1)-\ell+i$ for $ 1 \leq i \leq \ell$. Thus the size of
$\lambda$ is  $| \lambda |=\sum_{x\in
\beta(\lambda)}{x}-\binom{ |\beta(\lambda)| }{2}$.

(2) (\textbf{Abacus condition for t-core partitions.})
A partition $\lambda$ is a $t$-core partition if and only if for any
$x\in \beta(\lambda)$ with $x\geq t$,
we always have $x-t \in \beta(\lambda)$.
\end{lemma}

\begin{rem}\label{remark1}
An element $x\in\beta(\la)$ is called \emph{$t$-maximal} if $x+t\notin\beta(\la)$. Lemma \ref{betaset}(2) means that the $\beta$-set $\beta(\lambda)$ of a $t$-core partition $\lambda$ is determined by all $t$-maximal elements in $\beta(\la)$. Thus there is a bijection $\eta$ which sends each $t$-core partition $\la$ to $(n_1,n_2,\cdots,n_{t-1}):=\left(n_1(\la),n_2(\la),\cdots,n_{t-1}(\la)\right)\in \mathbb{N}^{t-1}$ such that $t(n_i-1)+i$ is maximal in $\beta(\la)$ if $n_i\geq 1$; and $i\notin \beta(\la)$ if $n_i$=0 for $1\leq i \leq t-1$. In this case, $
\beta(\la)=\bigcup_{i=1}^{t-1}\bigcup_{j=0}^{n_i-1}\{jt+i\}
$ and therefore $|\beta(\la)|=\sum_{i=1}^{t-1}n_i$.
\end{rem}

\begin{ex}
Let $\lambda=(5, 3,3, 2,1)$ be a $8$-core partition. Then $\beta(\lambda)=\{9,6,5,3,1\}$ and $\eta(\la)=(2,0,1,0,1,1,0)$.
\end{ex}

By Lemma \ref{betaset}(1), we have the following result.
\begin{lemma} \label{distinct}
The partition  $\lambda$ is a partition with distinct parts if and only if there doesn't exist $x,y\in \beta(\lambda)$ with $x-y=1.$
\end{lemma} 
\begin{proof}
Suppose that $\lambda = (\lambda_1, \lambda_2, \ldots, \lambda_\ell).$ Then by Lemma \ref{betaset}(1), we have $\lambda_i=\lambda_{i+1}$ if and only if $h(i,1)-\ell+i=h(i+1,1)-\ell+i+1,$ which is equivalent to $h(i,1)-h(i+1,1)=1.$ This implies the claim.
\end{proof}

%Now we turn to the generating function of $t$-core partitions with distinct parts. 
Let $\mathcal{A}_t=\{ x\in\mathbb{N}: 1 \leq x\leq t-1 \}$ for every $t\geq 2$. We say that a subset $B$ of $\mathcal{A}_t$ is \emph{nice} if $x-y\neq 1$ for any $x,y\in B$.
Let $\mathcal{B}_t$ be the set of nice subsets of $\mathcal{A}_t$ and $a_t=|\mathcal{B}_t|$ be the number of nice subsets of $\mathcal{A}_t$. Recall that $\mathcal{C}_t=\{ (x_1,x_2,\cdots,x_{t-1})\in\mathbb{N}^{t-1}: x_ix_{i+1}=0 \ \text{for}\ 1\leq i\leq t-2  \}$ for every $t\geq 2$.

\begin{ex}
Let $t=5$. We have $\mathcal{A}_5=\{1,2,3,4 \}$. The set of all nice subsets of $\mathcal{A}_5$ are $\mathcal{B}_5=\{  \{1\},\{2\},\{3\},\{4\},\{1,3\},\{1,4\},\{2,4\} \}$. 
Notice that $\mathcal{C}_5$ is determined by $\mathcal{B}_5$ in the following manner: $\mathcal{C}_5=\{ (x_1,0,0,0)\in\mathbb{N}^{4}\} \cup \{ (0,x_2,0,0)\in\mathbb{N}^{4}\} \cup \{ (0,0,x_3,0)\in\mathbb{N}^{4}\} \cup \{ (0,0,0,x_4)\in\mathbb{N}^{4}\} \cup \{ (x_1,0,x_3,0)\in\mathbb{N}^{4}\} \cup \{ (x_1,0,0,x_4)\in\mathbb{N}^{4}\} \cup \{ (0,x_2,0,x_4)\in\mathbb{N}^{4}\}$. 
\end{ex}

Let ${CD}_t$ be the set of $t$-core partitions with distinct parts and $cd_t(n)$ be the number of partitions in ${CD}_t$ with size $n$. Recall that $\eta$ is defined in Remark \ref{remark1}. 

\begin{theorem}\label{th:size}
The function  $\eta$ gives a bijection between the sets ${CD}_t$ and $\mathcal{C}_t$. If $\eta(\la)=(n_1,n_2,\cdots,n_{t-1})$ for some $t$-core partition $\la$ with distinct parts, then 
\begin{equation}\label{eq:1}
|\la|=\sum_{i=1}^{t-1}\left(in_i+t\binom{n_i}{2}\right)-\binom{\sum_{i=1}^{t-1}n_i}{2}.
\end{equation}
\end{theorem}
\begin{proof}
By Lemma \ref{distinct} we know if $\la$ is a $t$-core partition with distinct parts, then $i$ and $i+1$ couldn't be  in $\beta(\la)$ at the same time for $1\leq i \leq t-2$. Also we know $i\in \beta(\la)$ iff $n_i(\la)\geq 1$. Then the bijection between ${CD}_t$ and $\mathcal{C}_t$ is guaranteed by Remark \ref{remark1}. By the definition of $\eta$, we know $\eta(\la)=(n_1,n_2,\cdots,n_{t-1})$ means
$$
\beta(\la)=\bigcup_{i=1}^{t-1}\bigcup_{j=0}^{n_i-1}\{jt+i\}.
$$
Therefore by Lemma \ref{betaset}(1) we derive \eqref{eq:1}.
\end{proof}

Now we are ready to prove Theorem \ref{distinct*}.

\begin{proof}[Proof of Theorem \ref{distinct*}]
The formula \eqref{eq:2} is a direct corollary of Theorem \ref{th:size}.

In particular, when $t=2$, we have $\mathcal{C}_2=\mathbb{N}$. Then by \eqref{eq:2} we obtain
$$
\sum_{n\geq 0}cd_2(n)q^n=\sum_{n_1\geq 0}q^{n_1+2\binom{n_1}{2}-\binom{n_1}{2}}=\sum_{n\geq 0}q^{\binom{n+1}{2}}.
$$

When $t=3$, We have 
$$ 
\mathcal{C}_3=\{ (x_1,0)\in\mathbb{N}^{2}:x_1\geq 1  \} \cup
\{ (0,x_2)\in\mathbb{N}^{2}: x_2\geq 0  \} .
$$
By \eqref{eq:2} we obtain
\begin{align*}
\sum_{n\geq 0}cd_3(n)q^n
&=\sum_{n_1\geq 1}q^{n_1+3\binom{n_1}{2}-\binom{n_1}{2}}+\sum_{n_2\geq 0}q^{2n_2+3\binom{n_2}{2}-\binom{n_2}{2}}
\\&=\sum_{n\geq 1}q^{n^2}+\sum_{n\geq 0}q^{2\binom{n+1}{2}}.
\end{align*}

When $t=4$, We have 
$$ 
\mathcal{C}_4=\{ (0,x_2,0)\in\mathbb{N}^{2}:x_2\geq 1  \} \cup
\{ (x_1,0,x_3)\in\mathbb{N}^{2}: x_1\geq 0, x_3\geq 0  \} .
$$
By \eqref{eq:2} we obtain
\begin{align*}
\sum_{n\geq 0}cd_4(n)q^n
&=\sum_{n_2\geq 1}q^{2n_2+4\binom{n_2}{2}-\binom{n_2}{2}}+\sum_{n_1\geq 0}\sum_{n_3\geq 0}q^{n_1+3n_3+4\binom{n_1}{2}+4\binom{n_3}{2}-\binom{n_1+n_3}{2}}
\\&=\sum_{n\geq 1}q^{\frac{n(3n+1)}2}+\sum_{n\geq 0}\sum_{m\geq 0}q^{\frac{n(3n-1)+3m(m+1)-2mn}2}.
\end{align*}

\end{proof}

\bigskip

\section{$(t, t + 1)$-core partitions with distinct parts}

In this section we focus on $(t, t + 1)$-core partitions with distinct parts.
We have the following characterization for $\beta$-sets of
$(t,t+1)$-core partitions.

\begin{lemma} \label{set+1}
Let $t\geq 2$  be a positive integer. Suppose that $\lambda$ is a
$(t,t+1)$-core partition.  Then we have  $$\beta(\lambda)\subseteq \  \bigcup_{1\leq k \leq t-1} \{ x\in\mathbb{N}: (k-1)(t+1)+1\leq x \leq kt-1\}.$$
\end{lemma}
\begin{proof} 
By Lemma \ref{betaset}(2)  we have $at+b(t+1)\notin \beta(\lambda)$  for every $a,b \geq 0$.   Then
$$\beta(\lambda)\subseteq \mathbb{N}\setminus \{ at+b(t+1): a,b \geq 0  \} =  \bigcup_{1\leq k \leq t-1} \{ x\in\mathbb{N}: (k-1)(t+1)+1\leq x \leq kt-1\}.$$ \end{proof}

\begin{lemma} \label{setdis}
Let $t\geq 2$  be a positive integer. Suppose that $\lambda$ is a
$(t,t+1)$-core partition with distinct parts.  Then  $$\beta(\lambda)\subseteq \mathcal{A}_t= \{ x\in\mathbb{N}: 1 \leq x\leq t-1  \}.$$
\end{lemma}
\begin{proof} 
By Lemma \ref{betaset}(2)  we have $t,t+1\notin \beta(\lambda)$ since $0\notin \beta(\la)$.  For $x\geq t+2$, if $x\in\beta(\lambda),$ by Lemma \ref{betaset}(2) we know $x-t,x-(t+1)\in \beta(\lambda)$. But by Lemma \ref{distinct} we know this is impossible since $\lambda$ is a partition with distinct parts. Then we know  $x\notin \beta(\lambda)$ and thus $\beta(\lambda)$ is a subset of  $\mathcal{A}_t$. 
\end{proof}

Now we are ready to prove our main result Theorem \ref{main}.

\begin{proof}[Proof of Theorem \ref{main}]  %Let $\mathcal{A}_s=\{ x\in\mathbb{N}: 1 \leq x\leq s-1 \}$ for every $s\in \mathbb{N}.$ We say that a subset $B$ of $\mathcal{A}_s$ is \emph{nice} if $x-y\neq 1$ for any $x,y\in B$. Let $\mathcal{B}_s$ be the set of nice subsets of $\mathcal{A}_s$ and $a_s=\sum_{B\in \mathcal{B}_s}1=|\mathcal{B}_s|$ be the number of  nice subsets of $\mathcal{A}_s$.
(1) By Lemmas \ref{betaset}(2), \ref{distinct} and \ref{setdis} we know  a partition $\lambda$ is a $(t, t + 1)$-core partition with distinct parts if and only if $\beta(\lambda)$ is a nice subset of $\mathcal{A}_t$. 
 Thus the number of $(t, t + 1)$-core partitions with distinct parts equals the 
number $a_t$ of nice subsets of $\mathcal{A}_t$. Notice that $a_2=2=F_3$, $a_3=3=F_4$. When $t\geq 4$, suppose that $B\in \mathcal{B}_t$.  If  $t-1\in B$, then  $t-2\notin B$. This means that 
$$|\{B\in \mathcal{B}_t:t-1\in B\}|=|\mathcal{B}_{t-2}|=a_{t-2}.$$ Also we have $$|\{B\in \mathcal{B}_t:t-1\notin B\}|=| \mathcal{B}_{t-1}|=a_{t-1}.$$ Thus $a_t=a_{t-1}+a_{t-2}.$  Then we have $a_t=F_{t+1}$ for $t\geq 2$, which means that the number of $(t, t + 1)$-core partitions with distinct parts equals $F_{t+1}$. 

(2) Suppose that $\lambda$ is a $(t, t + 1)$-core partition with distinct parts such that $\beta(\lambda)=\{x_1,x_2,\ldots,x_k\}.$ 
In (1) we already know $\beta(\lambda)$ is a nice subset of $\mathcal{A}_t$. Thus 
\begin{align*}|\lambda|&=\sum_{i=1}^kx_i-\binom{k}{2}\leq \sum_{i=1}^k(t+1-2i)-\binom{k}{2}=-\frac32(k-\frac{2t+1}{6})^2+\frac{(2t+1)^2}{24}.
\end{align*}

When $t=3n$ for some integer $n$, we obtain $$|\lambda|\leq -\frac32(k-\frac{6n+1}{6})^2+\frac{(6n+1)^2}{24}\leq \frac{3n^2}2+\frac{n}{2}.$$

When $t=3n+1$ for some integer $n$, we obtain $$|\lambda|\leq -\frac32(k-\frac{6n+3}{6})^2+\frac{(6n+3)^2}{24}\leq \frac{3n^2}2+\frac{3n}{2}.$$

When $t=3n+2$ for some integer $n$, we obtain $$|\lambda|\leq -\frac32(k-\frac{6n+5}{6})^2+\frac{(6n+5)^2}{24}\leq \frac{3n^2}2+\frac{5n}{2}+1.$$

Finally, in each case we always obtain $$\lambda\leq [\frac13\binom{t+1}2].$$

(3)  By (2) we know, if $\lambda$ is  a $(t, t + 1)$-core partition with distinct parts which has the largest size, then its $\beta$-set must be  $\beta(\lambda)=\{t-1,t-3,\ldots,t-(2k-1)\}$  for some integer $k$.  When $t=3n$ for some integer $n$, $\lambda$ has the largest size $[\frac13\binom{t+1}2]$ if and only if  $k=n$; when $t=3n+1$ for some integer $n$, $\lambda$ has the largest size $[\frac13\binom{t+1}2]$ if and only if  $k=n$ or $n+1$; when $t=3n+2$ for some integer $n$, $\lambda$ has the largest size $[\frac13\binom{t+1}2]$ if and only if  $k=n+1$.
Therefore we prove the claim.

(4)  First we introduce some sequences. For every $t\geq 2$, let $$b_t=\sum_{B\in \mathcal{B}_t}|B|, \ \  c_t=\sum_{B\in \mathcal{B}_t}|B|^2,$$ $$d_t=\sum_{B\in \mathcal{B}_t}\sum_{x\in B}x, \ \ e_t=d_t-\sum_{B\in \mathcal{B}_t}\binom{|B|}2,$$ and   $$\phi_t=\sum\limits_{\substack{i+j=t\\ i,j\geq 1}}F_iF_j,  \ \ \psi_t=\sum\limits_{\substack{i+j+k=t\\ i,j,k\geq 1}}F_iF_jF_k.$$

Then by Lemmas \ref{betaset}, \ref{distinct} and \ref{setdis} we obtain 
$e_t$ equals the total sum of the sizes of all $(t, t + 1)$-core partitions with distinct parts. Thus we just need to show that $e_t=\psi_{t+1}$ for $t\geq 2$.

When $t\geq 4$, suppose that $B\in \mathcal{B}_t$.  If  $t-1\in B$, then  $t-2\notin B$. Therefore 
\begin{align*}b_t&=\sum\limits_{\substack{B\in \mathcal{B}_t\\ t-1\notin B}}|B|+\sum\limits_{\substack{B\in \mathcal{B}_t\\ t-1\in B}}|B|=\sum\limits_{\substack{B\in \mathcal{B}_{t-1}}}|B|+\sum\limits_{\substack{B\in \mathcal{B}_{t-2}}}(|B|+1)\\&=b_{t-1}+b_{t-2}+a_{t-2}=b_{t-1}+b_{t-2}+F_{t-1}.\end{align*}

Similarly we have 
\begin{align*}c_t&=\sum\limits_{\substack{B\in \mathcal{B}_t\\ t-1\notin B}}|B|^2+\sum\limits_{\substack{B\in \mathcal{B}_t\\ t-1\in B}}|B|^2=\sum\limits_{\substack{B\in \mathcal{B}_{t-1}}}|B|^2+\sum\limits_{\substack{B\in \mathcal{B}_{t-2}}}(|B|+1)^2\\&=c_{t-1}+c_{t-2}+2b_{t-2}+F_{t-1}
\end{align*}
and 
\begin{align*}d_t&=\sum\limits_{\substack{B\in \mathcal{B}_t\\ t-1\notin B}}\sum_{x\in B}x+\sum\limits_{\substack{B\in \mathcal{B}_t\\ t-1\in B}}\sum_{x\in B}x=\sum\limits_{\substack{B\in \mathcal{B}_{t-1}}}\sum_{x\in B}x+\sum\limits_{\substack{B\in \mathcal{B}_{t-2}}}(t-1+\sum_{x\in B}x)\\&=d_{t-1}+d_{t-2}+(t-1)F_{t-1}.
\end{align*}

Notice that 
$$e_t=d_t-\sum_{B\in \mathcal{B}_t}\binom{|B|}2=d_t-\frac12(c_t-b_t),$$
which means that 

$$e_t-e_{t-1}-e_{t-2}=(t-1)F_{t-1}-b_{t-2}.$$

Since  $e_2=\psi_3=1,$ $e_3=\psi_4=3,$ to show that $e_t=\psi_{t+1}$ for $t\geq 2$, we just need to show that
\begin{align}\label{eq:epsi}
\psi_{t+1}-\psi_{t}-\psi_{t-1}=(t-1)F_{t-1}-b_{t-2}
\end{align} for $t\geq 4.$

Notice that $F_0=0,\ F_1=1.$  We have 
\begin{align*}\psi_{t+1}&=\sum\limits_{\substack{i+j+k=t+1\\ i,j,k\geq 1}}F_iF_jF_k=\sum\limits_{\substack{i+j+k=t+1\\ j,k\geq 1\\i\geq 2}}F_iF_jF_k+\sum\limits_{\substack{j+k=t\\ j,k\geq 1}}F_jF_k\\&=\sum\limits_{\substack{i+j+k=t+1\\ j,k\geq 1\\i\geq 2}}F_{i-1}F_jF_k+\sum\limits_{\substack{i+j+k=t+1\\ j,k\geq 1\\i\geq 2}}F_{i-2}F_jF_k+\phi_{t}\\&=\sum\limits_{\substack{i'+j+k=t\\ i',j,k\geq 1}}F_{i'}F_jF_k+\sum\limits_{\substack{i'+j+k=t-1\\ i',j,k\geq 1}}F_{i'}F_jF_k+\phi_{t}\\&=
\psi_{t}+\psi_{t-1}+\phi_{t}.\end{align*}

Thus (\ref{eq:epsi}) is equivalent to 
\begin{align}\label{eq:ephi}(t-1)F_{t-1}-b_{t-2}=\phi_{t}.\end{align}

Notice that  (\ref{eq:ephi}) is true for $t=4,5$. Also we have 
\begin{align*}\ & \ \ \ \    (t-1)F_{t-1}-b_{t-2}-\bigl((t-2)F_{t-2}-b_{t-3}\bigr)-\bigl((t-3)F_{t-3}-b_{t-4}\bigr)\\&=  (t-1)F_{t-1}-(t-2)F_{t-2}-(t-3)F_{t-3}-F_{t-3}\\&=  (t-1)F_{t-1}-(t-2)F_{t-1}=F_{t-1}\end{align*}

and \begin{align*}\phi_{t}&=\sum\limits_{\substack{i+j=t\\ i,j\geq 1}}F_iF_j=\sum\limits_{\substack{i+j=t\\ j\geq 1\\i\geq 2}}F_iF_j + F_{t-1} =\sum\limits_{\substack{i+j=t\\ j\geq 1\\i\geq 2}}F_{i-1}F_j+\sum\limits_{\substack{i+j=t\\ j\geq 1\\i\geq 2}}F_{i-2}F_j + F_{t-1} \\&=\sum\limits_{\substack{i'+j=t-1\\ i',j\geq 1}}F_{i'}F_j+\sum\limits_{\substack{i'+j=t-2\\ i',j\geq 1}}F_{i'}F_j+F_{t-1}=
\phi_{t-1}+\phi_{t-2}+F_{t-1}.\end{align*}

Now we obtain (\ref{eq:ephi}) is true for $t\geq 4$. This implies that $e_t=\psi_{t+1}$ for $t\geq 2$. Therefore 
the total sum of the sizes of all $(t, t + 1)$-core partitions with distinct parts is
$$e_t=\psi_{t+1}=\sum\limits_{\substack{i+j+k=t+1\\ i,j,k\geq 1}}F_iF_jF_k.$$
Then by (1) the average size of these partitions is $$\sum\limits_{\substack{i+j+k=t+1\\ i,j,k\geq 1}}\frac{F_iF_jF_k}{F_{t+1}}.$$
\end{proof}

\section{Acknowledgements}
The author is supported  by Grant [PP00P2\_138906] of the Swiss National Science Foundation and the Post-doctoral Fellowship from LABEX of the University of Strasbourg. The author would like to thank Amdeberhan for introducing his conjecture and Straub for pointing out his results on the enumeration of $(t, nt-1)$-core partitions with distinct parts.


\begin{thebibliography}{1}



\bibitem{tamd}
T. Amdeberhan, Theorems, problems and conjectures, preprint; {\tt  arXiv:1207.4045v6}.
 
 \bibitem{tamd1}
T. Amdeberhan and E. Leven, Multi-cores, posets, and lattice paths,  Adv. in Appl. Math. $71 (2015)$, $1-13$.

\bibitem{and}
J. Anderson, Partitions which are simultaneously $t_1$- and
$t_2$-core, Disc. Math. $248(2002),\ 237-243$.

\bibitem{AHJ}
D. Armstrong, C.R.H. Hanusa, and B. Jones, Results and conjectures on simultaneous core partitions,
European J. Combin. $41 (2014),\ 205-220$.

\bibitem{berge} C. Berge, \emph{Principles of Combinatorics}, Mathematics in Science and Engineering Vol. $72$, Academic
Press, New York, $1971$.


\bibitem{CHW}
W. Chen, H. Huang and L. Wang, Average size of a self-conjugate $s,t$-core partition,  Proc. Amer. Math. Soc. $144(4) (2016),\ 1391-1399$.


\bibitem{stanton2}
F. Garvan, D. Kim and D. Stanton, Cranks and $t$-cores, Inv. Math.
$101 (1990),\ 1-17.$




\bibitem{PJ}
P. Johnson, Lattice points and simultaneous core partitions, preprint; {\tt arXiv:1502.07934v2}.

\bibitem{Macdonald} I. G. Macdonald, \emph{Symmetric Functions and Hall Polynomials}, Oxford Mathematical Monographs, The Clarendon Press, Oxford University Press, New York, second edition, 1995. 

\bibitem{NS}
R. Nath and J. A. Sellers, A combinatorial proof of a relationship between maximal $(2k-1,2k+1)$ and $(2k-1,2k,2k+1)$-cores, Electron. J. Combin. $23(1) (2016)$, Paper $1.13$.
%\bibitem{frame}
%J. Frame,  G. Robinson,  and R. Thrall, The hook graphs of the
%symmetric group, Canad. J. Math. $6 (1954), \ 316-324.$

\bibitem{NS2}
R. Nath and J. A. Sellers,
Abaci structures of $(s, ms\pm1) $-core partitions,  Electron. J. Combin. $24(1) (2016)$, Paper $1.5$.

\bibitem{N3}
R. Nath, Symmetry in maximal $(s-1,s+1)$ cores. Integers $16 (2016)$, Paper No. $A18$.

\bibitem{ols}
J. Olsson and D. Stanton, Block inclusions and cores of partitions,
Aequationes Math. $74(1-2)(2007),\ 90-110$.

\bibitem{ec2} R. P. Stanley, \emph{Enumerative Combinatorics}, vol.~2, Cambridge University Press, New York/Cam\-bridge, 1999.

\bibitem{SZ}
R. P. Stanley and F. Zanello, The Catalan case of Armstrong's conjectures on simultaneous core partitions, SIAM J. Discrete Math., $29(1)(2015),\ 658-666$. 

\bibitem{stanton}
D. Stanton, Open positivity conjectures for integer partitions,
Trends Math. $2 (1999),\ 19 - 25$.

\bibitem{Straub}
A. Straub, Core partitions into distinct parts and an analog of Euler's theorem, European J. Combin. $57 (2016),\ 40-49$.

\bibitem{Wang}
V. Wang, Simultaneous core partitions: parameterizations and sums,  Electron. J. Combin. $23(1) (2016)$, Paper $1.4$.




\bibitem{Xiong1}
H. Xiong, On the largest size of $(t,t+1,..., t+p)$-core partitions, Disc. Math. $339(1) (2016)$,  $308-317$.

%\bibitem{Xiong} H. Xiong, The number of simultaneous core partitions, preprint; {\tt arXiv:1409.7038v2}.

\bibitem{YQJZ}
S. Yan, G. Qin, Z. Jin and R. Zhou, On $(2k+ 1, 2k+ 3) $-core partitions with distinct parts, preprint; {\tt arXiv:1604.03729}.

\bibitem{YZZ}
J. Yang, M. Zhong and R. Zhou, On the enumeration of $(s,s+1,s+2)$-core partitions, European J. Combin. $49 (2015),\ 203-217$.

\bibitem{ZZ}
A. Zaleski and D. Zeilberger, Explicit (polynomial!) expressions for the expectation, variance and higher moments of the size of a $(2n+ 1, 2n+ 3)$-core partition with distinct parts, preprint; {\tt arXiv:1611.05775}.

\bibitem{Za}
A. Zaleski,
Explicit expressions for the moments of the size of an $(s,s+1)$-core partition with distinct parts,
Adv. in Appl. Math. $84 (2017),\ 1-7$. 
\end{thebibliography}
\end{document}